\documentclass[12pt,reqno]{amsart}
\usepackage{amsmath, amssymb, amsfonts, xcolor}
\usepackage{hyperref}
\numberwithin{equation}{section}
\usepackage{amsthm}
\newtheorem{thm}{Theorem}[section]

\newtheorem{prop}[thm]{Proposition}
\newtheorem{lem}[thm]{Lemma}
\newtheorem{conj}[thm]{Conjecture}
\newtheorem{dfn}[thm]{Definition}

\newtheorem{cor}[thm]{Corollary}

\usepackage[a4paper, top = 1.2in, bottom = 1.2in, left = 1.1in, right = 1.2in]{geometry}

\numberwithin{equation}{section}

\newcommand{\F}{\mathbb{F}}
\newcommand{\N}{\mathbb{N}}
\newcommand{\Q}{\mathbb{Q}}

\newcommand{\Z}{\mathbb{Z}}

\newcommand{\f}{\mathbf{f}}

\newcommand{\mcO}{\mathcal{O}}

\newcommand{\mfm}{\mathfrak{m}}
\newcommand{\mfn}{\mathfrak{n}}

\newcommand{\mfp}{\mathfrak{p}}
\newcommand{\mfq}{\mathfrak{q}}
\newcommand{\mfP}{\mathfrak{P}}

\newcommand{\GL}{\mathrm{GL}}

\newcommand{\Gal}{\mathrm{Gal}}

\def\1{1\!\!1}
\newcommand{\mrm}[1]{\mathrm{#1}}

\title[On the solutions of Diophantine equations over $K$]{On the solutions of $x^p+y^p=2^rz^p$, $x^p+y^p=z^2$ over totally real fields}

\author[N. Kumar]{Narasimha Kumar}
\address[N. Kumar]{Department of Mathematics, Indian Institute of Technology Hyderabad, Kandi, Sangareddy 502285, INDIA.}
\email{narasimha@math.iith.ac.in}

\author[S. Sahoo]{Satyabrat Sahoo}
\address[S. Sahoo]{Department of Mathematics, Indian Institute of Technology Hyderabad, Kandi, Sangareddy 502285, INDIA.}
\email{ma18resch11004@iith.ac.in}

\keywords{Diophantine equations, $S$-unit equation, Semi-stability, Irreducibility of Galois representations, Modularity of elliptic curves, Level lowering}
\subjclass[2010]{Primary 11D41; Secondary 11F80, 11R04, 11R80}
\date{\today}

\begin{document}
	\maketitle
	\begin{abstract}
		In this article, we study the non-trivial primitive solutions of a certain type for the Diophantine equations $x^p+y^p=2^rz^p$ and $x^p+y^p=z^2$ of prime exponent $p$, $r \in \N$, over a totally real field $K$. Then for $r=2,3$, we study the non-trivial primitive solutions over $\mcO_K$ for the equation $x^p+y^p=2^rz^p$ of prime exponent $p$. Finally, we give several purely local criteria for $K$ such that the equation $x^p+y^p=2^rz^p$ has no non-trivial primitive solutions over $\mcO_K$.
	\end{abstract}

%
	\section{Introduction}
	Throughout this article, $K$ denotes a totally real number field. Let 
	$P$, $\mathbb{P}$ denote $\mrm{Spec}(\mcO_K)$, $\mrm{Spec}(\Z)$, respectively. In the literature, the thrust to understand the Diophantine equations, especially the Fermat type equations over number fields, has a long and enthralling history. Though there were many approaches available, the cyclotomic and Mordell-Weil methods were proved to be more successful to understand these equations. In fact, the modularity of elliptic curves was crucial in the proof of Fermat's Last Theorem by Wiles (cf.~\cite[Theorem 0.5]{W95}).

	In~\cite[Theorem 1.3]{JM04}, Jarvis and Meekin show that the Fermat equation $x^n+y^n=z^n$ of exponent $n\in \N$ has no non-trivial solutions over $\Z[\sqrt{2}]$ for $n\geq 4$. In \cite{Tur18}, \cite{Tur20}, \c{T}urca\c{s} studied the Fermat equation $x^p+y^p=z^p$ of exponent $p\in \mathbb{P}$, over imaginary quadratic fields of class number $1$.
	
	In~\cite[Theorem 3]{FS15}, Freitas and Siksek show that the asymptotic FLT holds for the Fermat equation $x^p+y^p=z^p$ over $K$, i.e., there exists a constant $B_K>0$ (depends on $K$) such that for primes $p >B_K$, the equation $x^p+y^p=z^p$ of exponent $p$ has no non-trivial solutions over $K$. 
    The proof depends upon certain explicit bounds on the solutions of $S$-unit equation~\eqref{S_K-unit solution}. 
	In \cite[Theorem 1.1]{SS18}, \c{S}eng\"{u}n and Siksek, extending the work in~\cite{FS15}, show that the asymptotic FLT holds for the Fermat equation $x^p+y^p=z^p$ over number fields. In~\cite[Theorem 1]{D16}, Deconinck extended the work in~\cite{FS15} to the Fermat equation $Ax^p+By^p=Cz^p$ with $2 \nmid ABC$. Later, in~\cite{KO20}, Kara and Ozman extended 
	the work in~\cite{D16} to number fields.

	In this article, we study the non-trivial primitive solutions (over $K$) of the Diophantine equation $x^p+y^p=2^rz^p$ with $p \in \mathbb{P},r \in \N$. Note that this equation is complementary to the one considered by Deconinck. We also study a similar question for the equation $x^p+y^p=z^2$ of exponent $p\in \mathbb{P}$ over $K$. 
	
	\begin{dfn}
    We say that the Diophantine equation $x^p+y^p=2^rz^p$ (resp., $x^p+y^p=z^2$) has no asymptotic solution in $S \subseteq \mcO_K^3$, if there exists a constant $B_{K,r}$ (resp., $B_K$) depending on $K,r$ (resp., on $K$) such that for primes $p >B_{K,r}$ (resp., $B_K$), the equation $x^p+y^p=2^rz^p$ (resp., $x^p+y^p=z^2$) of exponent $p$ has no non-trivial primitive solutions in $S$.
	\end{dfn}

	In~\cite[Theorem 3]{R97}, Ribet shows that there are no non-zero integer solutions to the equation $x^p+2^ry^p+z^p=0$ of exponent $p\in \mathbb{P}$ with $1\leq r <p$. In~\cite{DM97}, Darmon and Merel show that the equation $x^n+y^n=2z^n$ of exponent $n\in \N$ has no non-trivial primitive integer solutions for $n \geq 3$.
	
	In this article, we first show that the equation $x^p+y^p=2^rz^p$ of exponent $p\in \mathbb{P}$, with $r \in \N$, has no asymptotic solution in $W_K$, where $W_K$ is as in Definition~\ref{definition for W_K} (cf.\ Theorem~\ref{main result for x^p+y^p=2^rz^p}). Later, for $r=2,3$, we show that the   equation $x^p+y^p=2^rz^p$ of exponent $p\in \mathbb{P}$ has no asymptotic solution in $\mcO_K^3$ (cf. Theorem~\ref{main result1 for x^p+y^p=2^rz^p}). The proofs of Theorem~\ref{main result for x^p+y^p=2^rz^p} and Theorem~\ref{main result1 for x^p+y^p=2^rz^p} depend on certain explicit bounds on the solutions of $S_K$-unit equation~\eqref{S_K-unit solution}.

	
	In~\cite{DM97}, Darmon and Merel show that the equation $x^n+y^n=z^2$ of exponent $n\in \N$ has no non-trivial primitive integer solutions for $n\geq 4$. In~\cite[Theorem 1.1]{IKO20}, I\c{s}ik, Kara and Ozman show that, for any totally real field $K$ with narrow class number $h_K^+=1$ and \textbf{if there exists $\mathfrak{P} \in P$ over $2$ with residual degree $1$}, then the equation $x^p+y^p=z^2$ of exponent $p\in \mathbb{P}$ has no asymptotic solution of certain type over $K$. This proof depends upon the certain explicit bounds on the solutions of the $S_K$-unit equation~\eqref{S_L unit solution for (p,p,2)}. In~\cite[Theorem 6.1]{IKO22}, they extended this work to number fields.

%

In this article, we relax the assumption on the existence of $\mfP \in P$ over $2$ with residual degree 1 and $h_K^+=1$ in~\cite[Theorem 1.1]{IKO20}(cf. Theorem~\ref{main result1 for (p,p,2)}). More precisely, we show that the equation $x^p+y^p=z^2$ of exponent $p\in \mathbb{P}$ has no asymptotic solution in $W_K^\prime$ (cf. Definition~\ref{definition for W_K'} for $W_K'$).
	This proof depends upon certain explicit bounds on the solutions of the $S_K$-unit equation~\eqref{S_L unit solution for (p,p,2)}. 
Recently, Mocanu generalized \cite[Theorem 1.1]{IKO20}, but with a different hypothesis (cf. \cite[Theorem 3]{M22}).
Namely, from the assumption $h_K^+=1$ to $\mathrm{Cl}_{S_K}(K)=1$. In this case, the proof depends upon some explicit bounds on the solutions of $\alpha +\beta =\gamma^2$, with $\alpha, \beta \in \mcO_{S_K}^\ast$ and $\gamma \in  \mcO_{S_K}$ (cf. \cite[page 3]{M22} for the definition of $\mathrm{Cl}_{S_K}(K)$, $\mcO_{S_K}$ and $\mcO_{S_K}^\ast$).

	In~\cite{FKS21}, the authors gave some purely local criteria for $K$ such that the asymptotic FLT holds for the Fermat equation $x^p+y^p=z^p$ over $K$. In the last section, we give several purely local criteria for $K$ such that the equation $x^p+y^p=2^rz^p$ (resp., $x^p+y^p=2^rz^p$ with $r=2,3$) of exponent $p\in \mathbb{P}$ has no asymptotic solution in $W_K$ (resp., in $\mcO_K^3$). 

In the proofs of Theorems~\ref{main result for x^p+y^p=2^rz^p},~\ref{main result1 for x^p+y^p=2^rz^p} and ~\ref{main result1 for (p,p,2)}, we use the modularity of the Frey curve $E$ attached to a non-trivial primitive solution, irreducibility of the residual Galois representation $\bar{\rho}_{E,p}$, semi-stable reduction of $E$ at odd primes and level lowering.
	
	\subsection{Structure of the article:}
	The article is organized as follows. In \S\ref{section for preliminary}, we collate all the preliminaries required to prove main results. In \S\ref{notations section for x^p+y^p=2^rz^p} and \S\ref{section for proof of main results x^p+y^p=2^rz^p} (resp., \S\ref{section for AFLT (p,p,2)} and \S\ref{section for the proof of main result1 for (p,p,2)}), we state and prove Theorem~\ref{main result for x^p+y^p=2^rz^p},~\ref{main result1 for x^p+y^p=2^rz^p} (resp., Theorem~\ref{main result1 for (p,p,2)}) for the equation $x^p+y^p=2^rz^p$ (resp., $x^p+y^p=z^2$) of exponent $p\in \mathbb{P}$.

 	\section*{Acknowledgments}  
 	The authors sincerely thank Prof. Nuno Freitas for his help in understanding the article~\cite{FS15}.

	\section{Preliminaries}
	\label{section for preliminary}
	In this section, we recall some preliminaries and known results required to prove the main theorems of this article. Throughout this article, we denote $\mcO_K$, $\mfn$ and $p$ for the ring of integers of $K$,
	an ideal of $\mcO_K$ and a rational prime, respectively.
	\subsection{On Modularity result}
	We first recall a modularity result (cf.~\cite[Theorem 5]{FLBS15}).
	\begin{thm}
		\label{modularity result of elliptic curve over totally real}
		Up to isomorphism over $\bar{K}$, there are only a finite number of elliptic curves $E$ over $K$ which are not modular. Further, if $K$ is real quadratic, then every elliptic curve over $K$ is modular.
	\end{thm}
	\subsection{Eichler-Shimura}
    We now state a standard conjecture for $K$, which is a
	generalization of the Eichler-Shimura theorem over $\Q$.
	\begin{conj}[Eichler-Shimura]
		\label{ES conj}
		Let $f$ be a Hilbert modular newform over $K$ of parallel weight $2$, level $\mfn$ and with rational eigenvalues. Then there exists an elliptic curve $E_f /K$ of conductor $\mfn$ having same $L$-function as $f$.
	\end{conj}
	In~\cite{B04}, Blasius based on the work of Hida~\cite{H81}, gave a partial answer to Conjecture~\ref{ES conj}. More precisely:
	
	\begin{thm}
		\label{partial result to ES conj}
		Let $f$ be a Hilbert modular newform over $K$ of parallel weight $2$, level $\mfn$ and with rational eigenvalues. Suppose that either $[K: \Q] $ is odd or there is a finite prime $\mfq$ of $K$ such that $v_\mfq(\mfn) = 1$. 
		Then there exists an elliptic curve $E_f /K$ of conductor $\mfn$ having same $L$-function as $f$.
	\end{thm}
 
	For any elliptic curve $E/K$, let $\bar{\rho}_{E,p} : G_K:=\Gal(\overline{K}/K) \rightarrow \mathrm{Aut}(E[p]) \simeq \GL_2(\F_p)$ be the residual Galois representation of $G_K$, acting on the $p$-torsion points of $E$. For any Hilbert modular newform $f$ over $K$ of weight $k$, level $\mfn$ and character $\Psi$ with coefficient field $\Q_f$, let
    $\bar{\rho}_{f, \lambda}: G_K \rightarrow \GL_2(\F_\lambda)$
    be the residual Galois representation attached to $f, \lambda$, where $\lambda \in \mrm{Spec}(\mcO_{\Q_f})$. 
    
    In \cite[Corollary 2.2]{FS15}, Freitas and Siksek provided a partial answer to Conjecture~\ref{ES conj} in terms of mod $p$ Galois representations. More precisely:

	\begin{thm}
		\label{FS partial result of E-S conj}
		Let $E$ be an elliptic curve over $K$ and $p$ be an odd prime.
		Suppose that $\bar{\rho}_{E,p}$ is irreducible and $\bar{\rho}_{E,p} \sim \bar{\rho}_{f,p}$ for some Hilbert modular newform $f$ over $K$ of parallel weight $2$, level
		$\mfn$ with rational eigenvalues.
		Let $\mfq \nmid p$ be a prime of $K$ such that
		\begin{enumerate}
			\item E has potentially multiplicative reduction at $\mfq$,
			\item $p| \# \bar{\rho}_{E,p}(I_\mfq)$ and $p \nmid \left(\mathrm{Norm}(K/\Q)(\mfq) \pm 1\right)$.
		\end{enumerate}
		Then there exists an elliptic curve $E_f /K$ of conductor $\mfn$ having same $L$-function as of $f$.
	\end{thm}

	\subsection{Irreducibility of mod $p$ representations of elliptic curves}
	\label{section for irreducibilty x^p+y^p=2z^p}
	In~\cite[Theorem 2]{FS15 Irred}, Freitas and Siksek gave a criterion for determining the irreducibility of 
	$\bar{\rho}_{E,p}$, for any elliptic curve $E$ over $K$. More precisely:
	
	\begin{thm}
		\label{irreducibility of mod $P$ representation}
		Let $K$ be a totally real Galois field. Then there exists an effective constant $C_K$ (depends on $K$) such that, if $p>C_K$ is a prime and $E/K$ is an elliptic curve over $K$ which is semi-stable at all $\mfq |p$, then $\bar{\rho}_{E,p}$ is irreducible.
	\end{thm}

	\subsection{Level lowering}
	Let $E/K$ be an elliptic curve of conductor $\mfn$. For a prime ideal $\mfq$ of $K$, let $\Delta_\mfq$ be the discriminant of a minimal local model of $E$ at $\mfq$. Let 
	
	\begin{equation}
		\label{conductor of elliptic curve}
		\mfm_p:= \prod_{ p|v_\mfq(\Delta_\mfq), \ \mfq ||\mfn} \mfq \text{\quad and \quad } \mfn_p:=\frac{\mfn}{\mfm_p}.
	\end{equation}
    In~\cite[Theorem 7]{FS15}, the authors talked about the level lowering of mod $p$ Galois representations attached to elliptic curves over $K$.
	More precisely, they prove
	\begin{thm}
		\label{level lowering of mod $p$ repr}
		Let $E/K$ be an elliptic curve of conductor
		$\mfn$. Let $p$ be a rational prime. Suppose that the following conditions hold:
		\begin{enumerate}
			\item  For $p \geq 5$, the ramification index $e(\mfq /p) < p-1$ for all $\mfq |p$ and $\Q(\zeta_p)^+ \nsubseteq K$;
			\item $E/K$ is modular and $\bar{\rho}_{E,p}$ is irreducible;
			\item $E$ is semi-stable at all $\mfq |p$ and $p| v_\mfq(\Delta_\mfq)$ for all $\mfq |p$.
		\end{enumerate}
		Then there exists a Hilbert modular newform $f$ of parallel weight $2$, level $\mfn_p$ and some prime $\lambda$ of $\Q_f$ such that $\lambda | p$ and $\bar{\rho}_{E,p} \sim \bar{\rho}_{f,\lambda}$.
	\end{thm}

	\section{Solutions of the Diophantine equations $x^p+y^p=2^rz^p$ over totally real fields $K$}
	\label{notations section for x^p+y^p=2^rz^p}
	In this section, we study the solutions of the following equation over $K$.
	\begin{equation}
		\label{x^p +y^p = 2^rz^p}
		x^p +y^p = 2^rz^p
	\end{equation}  
	with prime exponent $p >5$ and $r\in \N $.	
	\begin{dfn}[Trivial solution]
		We say a solution $(a, b, c)\in  \mcO_K^3$ to the equation~\eqref{x^p +y^p = 2^rz^p} of exponent $p$ is trivial, if $abc=0$, or $(a, b, c)\in \{(1,1,1), (-1,-1,-1)\}$ for $r=1$, otherwise non-trivial. Further, we call it as primitive if $a,b,c$ are pairwise co-prime.
	\end{dfn} 	
    For any $S \subseteq P := \mrm{Spec}(\mcO_K)$, let $\mcO_{S}:=\{\alpha \in K : v_\mfP(\alpha)\geq 0 \text{ for all } \mfP \in P \setminus S\}$ denote the ring of $S$-integers in $K$ and 
    $\mcO_{S}^*$ denote the units of $\mcO_{S}$. We refer to them
    as $S$-units. Let $S_K:=\{ \mfP \in P : \mfP|2 \}$
	and $U_K:=\{ \mfP \in S_K: (3, v_\mfP(2))=1 \}$.
	
\begin{dfn}
	\label{definition for W_K}
	  Let $W_K$ be the set of $(a, b, c)\in  \mcO_K^3$ satisfying the equation~\eqref{x^p +y^p = 2^rz^p} of exponent $p$ with $\mfP |abc$ for every $\mfP \in S_K$.
\end{dfn}

	\subsection{Main results}
	\label{section for main result}

    We now state the main results of this article.
	\begin{thm}
	\label{main result for x^p+y^p=2^rz^p}
		Let $K$ be a totally real field. Suppose, for every solution $(\lambda, \mu)$ to the $S_K$-unit equation
		\begin{equation}
			\label{S_K-unit solution}
			\lambda+\mu=1, \ \lambda, \mu \in \mcO_{S_K}^\ast,
		\end{equation}
		there exists some $\mfP \in S_K$ that satisfies
		\begin{equation}
			\label{assumption for main result x^p+y^p=2^rz^p}
			\max \left\{|v_\mfP(\lambda)|,|v_\mfP(\mu)| \right\}\leq 4v_\mfP(2).
		\end{equation}
		Then the equation~\eqref{x^p +y^p = 2^rz^p} of exponent $p$ has no asymptotic solution in $W_K$, i.e., there exists a constant $B_{K,r}$ (depends on $K,r$) such that for primes $p>B_{K,r}$, the equation~\eqref{x^p +y^p = 2^rz^p} of exponent $p$ has no non-trivial primitive solutions in $W_K$.
	\end{thm}
	We write $(ES)$ for ``either $[K: \Q]\equiv 1 \pmod 2$ or Conjecture \ref{ES conj} holds for $K$."
	\begin{thm}
	\label{main result1 for x^p+y^p=2^rz^p}
		Let $K$ be a totally real field satisfying the condition $(ES)$. 
		Suppose, for every solution $(\lambda, \mu)$ to the $S_K$-unit equation~\eqref{S_K-unit solution}
		there exists some $\mfP \in U_K$ that satisfies
		\begin{equation}
			\label{assumption for main result1 x^p+y^p=2^rz^p}
			\max \left\{|v_\mfP(\lambda)|,|v_\mfP(\mu)| \right\}\leq 4v_\mfP(2) \text{ and } v_\mfP(\lambda\mu)\equiv v_\mfP(2) \pmod 3.
		\end{equation}
		Then the equation~\eqref{x^p +y^p = 2^rz^p} of exponent $p$ with $r=2,3$ has no asymptotic solution in $\mcO_K^3$, i.e., there exists a constant $B_K$ (depends on $K$) such that for primes $p>B_K$, the equation~\eqref{x^p +y^p = 2^rz^p} of exponent $p$ with $r=2,3$ has no non-trivial primitive solutions in $\mcO_K^3$.
	\end{thm}
	
	\begin{cor}
		\label{cor to main result1 for x^p+y^p=2^rz^p}
		Let $K$ and $r$ be as in Theorem~\ref{main result1 for x^p+y^p=2^rz^p}.
		Then the conclusion of Theorem~\ref{main result1 for x^p+y^p=2^rz^p}
		remains valid even if we replace both the assumptions in~\eqref{assumption for main result1 x^p+y^p=2^rz^p} by
		\begin{equation}
			\label{assumption for cor to main result1 x^p+y^p=2^rz^p}
			\max \left\{|v_\mfP(\lambda)|,|v_\mfP(\mu)| \right\}= v_\mfP(2).
		\end{equation}
	\end{cor}
	
	\begin{proof}[Proof of Corollary~\ref{cor to main result1 for x^p+y^p=2^rz^p}]
		Let $\mfP \in U_K$ and $t:=\max \left\{|v_\mfP(\lambda)|,|v_\mfP(\mu)| \right\}>0$. Since $\lambda +\mu =1$, we get $v_\mfP(\lambda\mu)= -2t$ or $t$. Therefore, $v_\mfP(\lambda\mu) \equiv t \pmod 3$. By~\eqref{assumption for cor to main result1 x^p+y^p=2^rz^p}, we get $v_\mfP(\lambda\mu) \equiv  v_\mfP(2) \pmod 3$. Now, the proof follows from Theorem~\ref{main result1 for x^p+y^p=2^rz^p}. 
	\end{proof}

\section{Steps to prove Theorem~\ref{main result for x^p+y^p=2^rz^p} and Theorem~\ref{main result1 for x^p+y^p=2^rz^p}}
	\label{section for proof of main results x^p+y^p=2^rz^p}
	For any non-trivial primitive solution $(a, b, c) \in \mcO_K^3$ to the equation~\eqref{x^p +y^p = 2^rz^p} of exponent $p>5$, consider the Frey curve
	\begin{equation}
		\label{Frey curve for x^p +y^p = 2^rz^p}
		E=E_{a,b,c} : Y^2 = X(X-a^p)(X+b^p).
	\end{equation}
	with $c_4=2^4(a^{2p}+2^rb^pc^p),\ \Delta_E=2^{4+2r}(abc)^{2p}$ and $j_E=2^{8-2r} \frac{(a^{2p}+2^rb^pc^p)^3}{(abc)^{2p}},$ where $j_E$ (resp., $\Delta_E$) denote the $j$-invariant (resp., discriminant) of $E$.
	
	\subsection{Modularity of the Frey curve}
	We will now prove the modularity of the Frey curve $E$ given by \eqref{Frey curve for x^p +y^p = 2^rz^p} over $K$.
	
	\begin{thm}
		\label{modularity of Frey curve x^p+y^p=2^rz^p over K}
		Suppose $(a,b,c)\in \mcO_K^3$ be a non-trivial primitive solution to the equation~$\eqref{x^p +y^p = 2^rz^p}$ of exponent $p$. Let $E/K$ be the Frey curve attached to $(a,b,c)$(cf.~\eqref{Frey curve for x^p +y^p = 2^rz^p}). Then there exists a constant $A_{K,r}$ (depends on $K,r$) such that for primes $p >A_{K,r}$, $E/K$ is modular.  
	\end{thm}
	
	\begin{proof}
		Since $j_E= 2^{8-2r}\frac{(a^{2p}+b^{2p}+a^pb^p)^3}{(abc)^{2p}}$, $j(\lambda)=2^8\frac{(\lambda^2-\lambda+1)^3}{\lambda^2(\lambda-1)^2}$ for $\lambda= -\frac{b^p}{a^p}$.	
		By Theorem~\ref{modularity result of elliptic curve over totally real}, there exists finitely many $\bar{K}$-isomorphism classes of non-modular elliptic curves over $K$. 
		Let $j_1,...,j_t \in K$ be the $j$-invariants of those elliptic curves.  
		
		For each $i=1, 2, \ldots, t$, the equation $j(\lambda)=j_i$ has at most six solutions in $K$. Hence, there exists 
		$\lambda_1, \lambda_2, ..., \lambda_m \in K$ with $m\leq 6t$ such that $E$ is modular for all $\lambda \notin\{\lambda_1, \lambda_2, ..., \lambda_m\}$.
		If $\lambda= \lambda_k$ for some $k \in \{1, 2, \ldots, m \}$, then $\left(\frac{b}{a} \right)^p=-\lambda_k$ and $\left(\frac{c}{a} \right)^p=\frac{1-\lambda_k}{2^r}.$
		The above two equations determine $p$ uniquely and denote it by $p_k$. Suppose $p \neq q$ are primes such that $\left(\frac{b}{a} \right)^p=\left(\frac{b}{a} \right)^q$, which means $\left(\frac{b}{a}\right)$ is a root of unity. Similarly,
		$\left(\frac{c}{a}\right)$ is also a root of unity.	We get $b=\pm a$, $c=\pm a$, as $K$ is totally real. Since $(a,b,c)\in \mcO_K^3$ is primitive, this gives $a=\pm1, b=\pm1, c=\pm1$ and hence the solution $(a,b,c)$ is trivial. 
		The proof of Theorem~\ref{modularity of Frey curve x^p+y^p=2^rz^p over K} follows by taking $A_{K,r}=\max \{p_1,...,p_m\}$.
	\end{proof}	

	\subsection{Reduction type}
		In this section, we describe the reduction of the Frey curve $E:= E_{a,b,c}$ attached to a non-trivial primitive solution $(a,b,c)$ to the equation~\eqref{x^p +y^p = 2^rz^p} at $\mfq \in P$.

	\subsubsection{Reduction type at odd primes}
	The following lemma characterizes the type of reduction of the Frey curve $E$ 
	at primes $\mfq$ away from $S_K$.
	\begin{lem}
		\label{reduction away from S}
		Suppose $(a,b,c) \in \mcO_K^3$ be a non-trivial primitive solution to the equation~\eqref{x^p +y^p = 2^rz^p} of exponent $p$. Let $E := E_{a,b,c}$ be the associated Frey curve as in \eqref{Frey curve for x^p +y^p = 2^rz^p}. Then at all primes $\mfq \in P$ away from $S_K$, $E$ is minimal, semi-stable and satisfies $p | v_\mfq(\Delta_E)$. Let $\mfn$ be the conductor of $E$ and $\mfn_p$ be as in \eqref{conductor of elliptic curve}. Then
		\begin{equation}
			\label{conductor for Frey curve x^p +y^p = 2^rz^p}
			\mfn=\prod_{\mfP \in S_K}\mfP^{r_\mfP} \prod_{\mfq|abc,\ \mfq \in P \setminus S_K}\mfq,\ \mfn_p=\prod_{\mfP \in S_K}\mfP^{r_\mfP'},
		\end{equation}
		where $0\leq r_\mfP' \leq r_\mfP \leq 2+6v_\mfP(2)$. 
	\end{lem}
	
	\begin{proof}
		Let $\mfq \in P \setminus S_K$. If $\mfq \nmid \Delta_E$, then $E$ has good reduction at $\mfq$. Recall that $\Delta_E=2^{4+2r}(abc)^{2p}$.
		If $\mfq|\Delta_E$, then $\mfq | abc$. Since $(a,b,c)$ is primitive solution to the equation~\eqref{x^p +y^p = 2^rz^p} and $\mfq \nmid 2$, $\mfq$ divides precisely one of $a$, $b$ and $c$. 
		Recall that $c_4=2^4(a^{2p}+2^rb^pc^p)$, so $v_\mfq(c_4)=0$. Hence, $E$ is minimal and has multiplicative reduction at $\mfq$. Therefore, $E$ is semi-stable at $\mfq$. Since $v_\mfq(\Delta_E)=2p \left( v_\mfq(abc) \right)$, $p | v_\mfq(\Delta_E)$. 
		By~\eqref{conductor of elliptic curve}, we get $\mfq \nmid \mfn_p$ for all  $\mfq \in P \setminus S_K$.
		For $\mfP \in S_K$, $r_\mfP=v_\mfP(\mfn)\leq 2+6v_\mfP(2)$ (cf. \cite[Chapter IV, Theorem 10.4]{S94}).
	\end{proof}

	\subsubsection{Type of the reduction based on image of inertia}
    We now recall~\cite[Lemma 3.4]{FS15}, which is useful for determining the type of reduction of the Frey curve $E$ based on the image of inertia for $\bar{\rho}_{E,p}$.

 	\begin{lem}
		\label{criteria for potentially multiplicative reduction}
		Let $E/K$ be an elliptic curve. Let $p > 5$ be a prime and $\mfq \nmid p$ be a prime of $K$. Then $E$ has potentially multiplicative reduction at $\mfq$ i.e., $\ v_\mfq(j_E) < 0$ and $p \nmid v_\mfq(j_E)$ if and only if $p | \# \bar{\rho}_{E,p}(I_\mfq)$.
	\end{lem}

	The next lemma specifies the type of reduction of the Frey curve $E$ at primes $\mfq \nmid 2$.
	\begin{lem}
		\label{Type of reduction at q away from 2 and p x^p+y^p=2^rz^p}
		Let $(a,b,c)  \in \mcO_K^3$ be a non-trivial primitive solution to the equation~\eqref{x^p +y^p = 2^rz^p} of exponent $p > 5$ and let $E := E_{a,b,c}$ be the associated Frey curve as in \eqref{Frey curve for x^p +y^p = 2^rz^p}. For $\mfq \in P$, if $\mfq\nmid 2p$ then $p \nmid \#\bar{\rho}_{E,p}(I_\mfq)$.
	\end{lem}
	
	\begin{proof}
		By Lemma~\ref{criteria for potentially multiplicative reduction}, it is enough to show that
		either $v_\mfq(j_E) \geq 0$ or $p | v_\mfq(j_E)$. If $\mfq \nmid \Delta_E$, then $E$ has good reduction at $\mfq$, hence $v_\mfq(j_E)\geq 0$. If $\mfq | \Delta_E$ then $\mfq |abc$. Since $(a,b,c)$ is a primitive solution to the equation~\eqref{x^p +y^p = 2^rz^p} and $\mfq \nmid 2$, $\mfq$ divides exactly one of $a$, $b$ and $c$. Therefore,
		$v_\mfq (c_4)=0, v_\mfq(j_E)=-2p(v_\mfq(abc))$ and hence we are done.
	\end{proof}
	
	\subsubsection{Reduction type at primes $\mfP \in S_K$}
	The following lemma is quite useful.
	\begin{lem}
		\label{lemma for W_K}
		Let $(a,b,c) \in \mcO_K^3$ be a non-trivial primitive solution to the equation~\eqref{x^p +y^p = 2^rz^p} of exponent $p> [K: \Q]r$. 
		For $\mfP \in S_K$, if $\mfP |abc$, then $\mfP | c$. In particular, $\mfP \nmid ab.$
	\end{lem}
\begin{proof}
    If $\mfP \nmid c$, then $\mfP |ab$. So, either $\mfP |a$ or $\mfP |b$. Since $\mfP |2$ and $a^p +b^p = 2^rc^p$ with $r \in \N$, $\mfP$ divides both $a$ and $b$, which implies $\mfP^p | 2^rc^p$. This cannot happen by the unique factorization of ideals in $\mcO_K$ and $p> [K: \Q]r$.
\end{proof}
		

We now recall~\cite[Lemma 3.6]{FS15}, which is useful for determining the type of reduction of the Frey curve $E$ at primes $\mfP \in U_K$.
	\begin{lem}
		\label{3 divides discriminant}
		Let $E/K$ be an elliptic curve. Let $p\geq 3$ be a prime and $\mfP \in S_K$. Suppose $E$ has a potential good reduction at $\mfP$. Then $3 \nmid v_\mfP(\Delta_E)$ if and only if $3 | \#\bar{\rho}_{E,p}(I_\mfP)$.
	\end{lem}
	We will now discuss the type of reduction of the Frey curve $E:=E_{a,b,c}$ at  $\mfP \in S_K$ for $(a,b,c) \in W_K$ and at $\mfP \in U_K$ for $(a,b,c) \in \mcO_K^3$. More precisely:
	\begin{lem}
		\label{reduction on T and S}
		Let $\mfP \in S_K$. Suppose $(a,b,c) \in \mcO_K^3$ is a non-trivial primitive solution to the equation~\eqref{x^p +y^p = 2^rz^p} of exponent $p >  \max \left\{ |(4-r)v_\mfP(2)|, [K:\Q]r \right\}$ and let $E := E_{a,b,c}$ be the associated Frey curve as in \eqref{Frey curve for x^p +y^p = 2^rz^p}.
		\begin{enumerate}
			\item If $(a,b,c)\in W_K$, then $v_\mfP(j_E) <0$ and $p \nmid v_\mfP(j_E)$, equivalently $p | \#\bar{\rho}_{E,p}(I_\mfP)$.
			\item Let $r=2,3$. If $\mfP \in U_K$, then either $p | \#\bar{\rho}_{E,p}(I_\mfP)$ or $3 | \#\bar{\rho}_{E,p}(I_\mfP)$.
		\end{enumerate} 
	\end{lem} 
	
	\begin{proof}
    The $j$-invariant of $E$ is given by $j_E=2^{8-2r} \frac{(a^{2p}+2^rb^pc^p)^3}{(abc)^{2p}}$. By Lemma~\ref{lemma for W_K}, we get $v_\mfP(j_E)= 2\left( (4-r)v_\mfP(2)-pv_\mfP(c) \right)$.
		Since $p > |(4-r)v_\mfP(2)|$, $v_\mfP(j_E) <0$ and $p \nmid v_\mfP(j_E)$. By Lemma~\ref{criteria for potentially multiplicative reduction}, we get $p | \#\bar{\rho}_{E,p}(I_\mfP)$. This completes the proof of $(1)$.
		
		Let $\mfP \in U_K$. Suppose $(a,b,c)$ is a non-trivial primitive solution to the equation~\eqref{x^p +y^p = 2^rz^p} of exponent $p$ with $r=2,3$. Recall that $\Delta_E=2^{4+2r}(abc)^{2p}$. If $\mfP | a b c$, then by Lemma~\ref{lemma for W_K}, $\mfP |c$ and $ \mfP \nmid ab$. So by first part, $p | \#\bar{\rho}_{E,p}(I_\mfP)$. If $\mfP \nmid abc$, then $v_\mfP(j_E) = 4v_\mfP(2)$ (resp., $2v_\mfP(2)$), $v_\mfP(\Delta_E)= 8v_\mfP(2),$
		(resp., $10v_\mfP(2)$) for $r=2$ (resp., $r=3$). Since $\mfP \in U_K$,  $v_\mfP(j_E) \geq 0$ and $3 \nmid v_\mfP(\Delta_E)$ for $r=2,3$. Now, the proof of $(2)$ follows from Lemma~\ref{3 divides discriminant}.
	\end{proof}

	\subsection{Proofs of Theorem~\ref{main result for x^p+y^p=2^rz^p} and Theorem~\ref{main result1 for x^p+y^p=2^rz^p}}
	
	\subsubsection{Proof of Theorem~\ref{main result for x^p+y^p=2^rz^p}}
	The proof of this theorem depends on the following auxiliary result.	
	\begin{thm}
		\label{auxilary result x^p+y^p=2^rz^p}
		Let $K$ be a totally real field. Then there is a constant $B_{K,r}>0$ (depends on $K,r$) such that the following hold.
		Let $(a,b,c)\in W_K$ be a non-trivial primitive solution to the equation~\eqref{x^p +y^p = 2^rz^p} of exponent $p> B_{K,r}$ and let $E := E_{a,b,c}$ be the associated Frey curve as in \eqref{Frey curve for x^p +y^p = 2^rz^p}. Then there exists an elliptic curve $E^\prime/K$ such that:
		\begin{enumerate}
			\item $E^\prime/K$ has good reduction away from $S_K$ and has full $2$-torsion, i.e., $|E^\prime(K)[2]|=4$;
			\item $\bar{\rho}_{E,p} \sim\bar{\rho}_{E^\prime,p}$ and  $v_\mfP(j_{E^\prime})<0$ for all $\mfP \in S_K$.
		\end{enumerate}
			%
	\end{thm}
	
			%

		\begin{proof}[Proof of Theorem~\ref{auxilary result x^p+y^p=2^rz^p}] 
			By Lemma~\ref{modularity of Frey curve x^p+y^p=2^rz^p over K}, $E$ is modular for all primes $p>A_{K,r}$. By Lemma~\ref{reduction away from S}, the Frey curve $E$ is semi-stable away from $S_K$. If necessary, we can take the Galois closure of $K$ to ensure that   $\bar{\rho}_{E,p}$ is irreducible for all primes $p \gg 0$ (cf. Theorem~\ref{irreducibility of mod $P$ representation}).
			
			By Theorem~\ref{level lowering of mod $p$ repr}, there exist a Hilbert modular newform $f$ of parallel weight $2$, level $\mfn_p$ and some prime $\lambda$ of $\Q_f$ such that $\lambda | p$ and $\bar{\rho}_{E,p} \sim \bar{\rho}_{f,\lambda}$ for $p \gg 0$. 
			By allowing $p$ to be sufficiently large, we can assume $\Q_f=\Q$ (cf. ~\cite[\S 4]{FS15} for more details).
			
			Let $\mfP \in S_K$. Then $E$ has potential multiplicative reduction at $\mfP$ and $p | \#\bar{\rho}_{E,p}(I_\mfP)$
			for $p >\max \left\{ |(4-r)v_\mfP(2)|, [K:\Q]r \right\}$
			(cf. Lemma~\ref{reduction on T and S}). The existence of an elliptic curve $E_f$ of conductor $\mfn_p$
			then follows from Theorem~\ref{FS partial result of E-S conj} for all $p \gg 0$ (also excluding the primes $p \mid \left( \text{Norm}(K/\Q)(\mfP) \pm 1 \right)$). Therefore, $\bar{\rho}_{E,p} \sim \bar{\rho}_{E_f,p}$  for all primes $p>B_{K,r}$, where $B_{K,r}$ is the maximum of all the above implicit lower bounds.
			
			\begin{enumerate}
				\item Since the conductor of $E_f$ is $\mfn_p$ given in \eqref{conductor for Frey curve x^p +y^p = 2^rz^p}, $E_f$ has good reduction away from $S_K$. After enlarging $B_{K,r}$ by a sufficient amount and by possibly replacing $E_f$ with an isogenous curve, say $E^\prime$, we will find that $E^\prime/ K$ has full $2$-torsion. 
				This follows from~\cite[Proposition 15.4.2]{C07} and the fact that $E/K$ has all its points of order $2$ (cf.~\cite[\S 4]{FS15} for more details). The elliptic curve $E^\prime$ has good reduction away from $S_K$.
				
				\item Let $\mfP \in S_K$. Since $E_f$ is isogenous to $E^\prime$, $\bar{\rho}_{E,p} \sim \bar{\rho}_{E^\prime,p}$ for primes $p > B_{K,r}$. Now, by Lemma~\ref{reduction on T and S}, we get $p |\# \bar{\rho}_{E,p}(I_\mfP)= \# \bar{\rho}_{E^\prime,p}(I_\mfP)$. Finally, by Lemma~\ref{criteria for potentially multiplicative reduction}, we have $v_\mfP(j_{E^\prime})<0$.
			\end{enumerate}
		\end{proof}
		We now prove Theorem~\ref{main result for x^p+y^p=2^rz^p}.
		\begin{proof}[Proof of Theorem~\ref{main result for x^p+y^p=2^rz^p}]
			Let $B_{K,r}$ be as in Theorem~\ref{auxilary result x^p+y^p=2^rz^p}. Let $(a,b,c)\in W_K$ be a non-trivial primitive solution to the equation~\eqref{x^p +y^p = 2^rz^p} of exponent $p>B_{K,r}$. By Theorem~\ref{auxilary result x^p+y^p=2^rz^p}, there exists an elliptic curve $E^\prime/K$ having full $2$-torsion.
			Consequently, $E^\prime$ has a model of the type $E^\prime: Y^2 = (X-e_1)(X-e_2)(X-e_3)$, where $e_1, e_2, e_3$ are all distinct and their cross ratio $\lambda= \frac{e_3-e_1}{e_2-e_1} \in \mathbb{P}^1(K)-\{0,1,\infty\}$. In fact,  $E^\prime$  is isomorphic (over $\bar{K}$) to an elliptic curve $E_\lambda$ in the Legendre form:
			$$E_\lambda : y^2 = x(x - 1)(x-\lambda) \text{ for } \lambda \in \mathbb{P}^1(K)-\{0,1,\infty\}$$ 
			with
			\begin{equation}
				\label{j'-invariant of Legendre form}
				j_{E^\prime} = j(E_\lambda)= 2^8\frac{(\lambda^2-\lambda+1)^3}{\lambda^2(1-\lambda)^2}.
			\end{equation}
			Then the action of the symmetric group $S_3$ on $\{e_1,e_2,e_3\}$ can be extented to an action on $\mathbb{P}^1(K)-\{0,1,\infty\}$ via the cross ratio $\lambda= \frac{e_3-e_1}{e_2-e_1}$. Under the action of $S_3$, the orbit
			of $\lambda \in \mathbb{P}^1(K)-\{0,1,\infty\}$  is the set $\left \{\lambda, \frac{1}{\lambda}, 1-\lambda, \frac{1}{1-\lambda}, \frac{\lambda}{\lambda-1}, \frac{\lambda-1}{\lambda}\right \}$,
			which are called as the $\lambda$-invariants (cf. \cite[\S5]{FS15}).
			
			Let $\lambda\in K$ be any $\lambda$-invariant. Since $E^\prime$ has good reduction away from $S_K$, $j_{E^\prime} \in \mcO_{S_K}$. By \eqref{j'-invariant of Legendre form}, $\lambda\in K$ satisfies a monic polynomial of degree $6$ over $\mcO_{S_K}$, hence $\lambda \in \mcO_{S_K}$. Similarly for $\frac{1}{\lambda}$ as well. Therefore, $\lambda \in \mcO^*_{S_K}$ and $\mu:=1-\lambda \in \mcO^*_{S_K}$. Hence, $(\lambda, \mu)$ satisfy the $S_K$-unit equation~\eqref{S_K-unit solution}. 
			Rewriting \eqref{j'-invariant of Legendre form} in terms of $\lambda, \ \mu$, we get
			\begin{equation}
				\label{j' in terms of lambda and mu}
				j_{E^\prime}= 2^8\frac{(1-\lambda \mu)^3}{(\lambda \mu)^2}.
			\end{equation}
			
			By \eqref{assumption for main result x^p+y^p=2^rz^p}, there exists $\mfP \in S_K$ that satisfies $t:=\max \left\{|v_\mfP(\lambda)|,|v_\mfP(\mu)| \right\}\leq 4v_\mfP(2)$. 
			If $t=0$, then $v_\mfP(\lambda)= v_\mfP(\mu)=0$. Thus, $v_\mfP(j_{E^\prime})\geq 8v_\mfP(2)>0$, which contradicts Theorem~\ref{auxilary result x^p+y^p=2^rz^p}. 
			If $t>0$, then  $v_\mfP(\lambda)=v_\mfP(\mu)=-t$, or $v_\mfP(\lambda)=0$ and $v_\mfP(\mu)=t$, or $v_\mfP(\lambda)=t$ and $ v_\mfP(\mu)=0$. This implies $v_\mfP(\lambda \mu)=-2t$ or $t$. In all the cases, we have $v_\mfP(j_{E^\prime})\geq 8v_\mfP(2)-2t$, so $v_\mfP(j_{E^\prime})\geq 0 $, which contradicts Theorem~\ref{auxilary result x^p+y^p=2^rz^p}.
		\end{proof}
		
		\subsubsection{Proof of Theorem~\ref{main result1 for x^p+y^p=2^rz^p}.}
		The proof of this theorem depends on the following auxiliary result.	
		\begin{thm}
			\label{auxilary result1 x^p+y^p=2^rz^p}
			Let $K$ be a totally real field satisfying $(ES)$. Then there exists a constant $B_K$ (depends on $K$) such that the following hold. Let $(a,b,c)\in \mcO_K^3$ be a non-trivial primitive solution to the equation~\eqref{x^p +y^p = 2^rz^p} of exponent $p > B_K$ with $r=2,3$ and let $E := E_{a,b,c}$ be the associated Frey curve as in \eqref{Frey curve for x^p +y^p = 2^rz^p}. Then there exists an elliptic curve $E^\prime/K$ such that:
			\begin{enumerate}
				\item $E^\prime/K$ has good reduction away from $S_K$ and has full $2$-torsion, i.e., $|E^\prime(K)[2]|=4$ and $\bar{\rho}_{E,p} \sim\bar{\rho}_{E^\prime,p}$;
				\item for $\mfP \in U_K$,  either $v_\mfP(j_{E^\prime})<0$ or $3\nmid v_\mfP(j_{E^\prime})$.
			\end{enumerate}
		\end{thm}
		\begin{proof}
			Most of the proof is similar to that of Theorem~\ref{auxilary result x^p+y^p=2^rz^p} except 
			for the last part. Suppose $\mfP \in U_K$ and $r=2,3$. 
			\begin{itemize}
				\item If $p | \# \bar{\rho}_{E,p}(I_\mfP)= \# \bar{\rho}_{E^\prime,p}(I_\mfP)$, then $v_\mfP(j_{E^\prime})<0$,
				by Lemma~\ref{criteria for potentially multiplicative reduction}.
				\item If $p \nmid \# \bar{\rho}_{E,p}(I_\mfP)$, then $3 | \# \bar{\rho}_{E,p}(I_\mfP)=\# \bar{\rho}_{E^\prime,p}(I_\mfP)$, by Lemma~\ref{reduction on T and S}. If $v_\mfP(j_{E^\prime}) < 0$, then we are done. If $v_\mfP(j_{E^\prime}) \geq 0$, then by Lemma~\ref{3 divides discriminant}, we get $3\nmid v_\mfP(\Delta_{E^\prime})$. 
                Since $v_\mfP(j_{E^\prime}) = 3v_\mfP(c_4) -v_\mfP(\Delta_{E^\prime})$ and  $3\nmid v_\mfP(\Delta_{E^\prime})$, we have $3\nmid v_\mfP(j_{E^\prime})$.
			\end{itemize}
		\end{proof}
		\begin{proof}[Proof of Theorem~\ref{main result1 for x^p+y^p=2^rz^p}]
			By~\eqref{assumption for main result1 x^p+y^p=2^rz^p}, there exists $\mfP \in U_K$ such that $\max \left\{|v_\mfP(\lambda)|,|v_\mfP(\mu)| \right\}\leq 4v_\mfP(2)$ and $v_\mfP(\lambda\mu)\equiv v_\mfP(2) \pmod 3$. 
			Now, arguing as in the proof of Theorem~\ref{main result for x^p+y^p=2^rz^p} and by Theorem~\ref{auxilary result1 x^p+y^p=2^rz^p}, 
			we get $v_\mfP(j_{E^\prime}) \geq 0$, because $\max \left\{|v_\mfP(\lambda)|,|v_\mfP(\mu)| \right\}\leq 4v_\mfP(2)$. Since $j_{E^\prime}= 2^8\frac{(1-\lambda \mu)^3}{(\lambda \mu)^2}$, we have $v_\mfP(j_{E^\prime}) \equiv 8v_\mfP(2)-2v_\mfP(\lambda\mu) \pmod 3$. 
			Since $v_\mfP(\lambda\mu)\equiv v_\mfP(2) \pmod 3$, we have $v_\mfP(j_{E^\prime})\equiv 6v_\mfP(2) \pmod 3$ and hence $3 | v_\mfP(j_{E^\prime})$.
			Therefore, $v_\mfP(j_{E^\prime}) \geq0$ and $3 | v_\mfP(j_{E^\prime})$, which contradicts Theorem~\ref{auxilary result1 x^p+y^p=2^rz^p}.
		\end{proof}
		%

		\section{Solutions of the Diophantine equation $x^p +y^p = z^2$ over $K$}
		\label{section for AFLT (p,p,2)}
		In this section, we study the solutions of the following equation over $K$.
		\begin{equation}
			\label{p,p,2}
			x^p +y^p = z^2
		\end{equation}  
		with prime exponent $p>2$.
		
		\begin{dfn}[Trivial solution]
			We say a solution $(a,b,c) \in \mcO_K^3$ to the equation~\eqref{p,p,2} of exponent $p$, is trivial if $abc=0$, otherwise non-trivial. Further, we call it as primitive if $a,b,c$ are pairwise co-prime.
		\end{dfn} 
	Recall that $S_K=\{ \mfP \in P: \mfP|2 \}$. Define $T_K:=\{ \mfP \in S_K :\f(\mfP, 2)=1 \}$.
	\begin{dfn}
		\label{definition for W_K'}
			Let $W^\prime_K$ be the set of $(a, b, c)\in \mcO_K^3$ satisfying the equation~\eqref{p,p,2} of exponent $p$ with $\mfP |ab$ for every $\mfP \in S_K$. 
	\end{dfn}
		\subsection{Main result} 
		For any positive integer $m$, we define the $m$-Selmer group of $K$ and $S_K$ by $$K(S_K,m):=\{x\in K^*/(K^*)^m : v_\mfp(x) \equiv 0 \pmod m\  \forall \mfp \notin S_K \}.$$ 
		By ~\cite[\S 5.2.2 and 7.4]{C00},  $K(S_K,m)$ is a finite abelian group, for every $K$ and $m$. 
		Let $L=K(\sqrt{a})$ for $a \in K(S_K,2)$. Let $S_L$ be the set of all prime ideals of $L$ lying over primes of $S_K$.
			Consider $S_K$-unit (resp., $S_L$-unit) equation
			\begin{equation}
				\label{S_L unit solution for (p,p,2)}
				\lambda+ \mu=1, \ \lambda, \mu \in \mcO_{S_K}^\ast \ (\mrm{resp}.,
				\lambda, \mu \in \mcO_{S_L}^\ast).
			\end{equation}
			We now state the main result of this section.
			\begin{thm}
				\label{main result1 for (p,p,2)}
				Let $K$ be a totally real field. For each $a\in K(S_K,2)$, let $L=K(\sqrt{a})$. Suppose, for every solution $(\lambda, \mu)$ to the $S_K$-unit solution~\eqref{S_L unit solution for (p,p,2)}, there exists some $\mfP \in S_K$ that satisfies
				\begin{equation}
					\label{assumption for main result1 for T_K for (p,p,2)}
				\max \left\{|v_\mfP(\lambda)|,|v_\mfP(\mu)| \right\}\leq 4v_\mfP(2), 
				\end{equation}
				and for every solution $(\lambda, \mu)$ to the $S_L$-unit equation~\eqref{S_L unit solution for (p,p,2)}, there exists some $\mfP' \in S_L$ that satisfies
				\begin{equation}
					\label{assumption for main result1 for T_L for (p,p,2)}
					\max \left\{|v_{\mfP'}(\lambda)|,|v_{\mfP'}(\mu)| \right\} \leq 4v_{\mfP'}(2).
				\end{equation}
				Then the equation~\eqref{p,p,2} of exponent $p$ has no asymptotic solution in $W_K^\prime$, i.e., there exists a constant $B_K$ (depends on $K$) such that for primes $p>B_K$, the equation~\eqref{p,p,2} of exponent $p$ has no non-trivial primitive solutions in $W_K^\prime$.
			\end{thm}
			The theorem above can be thought of as a generalization of 
			~\cite[Theorem 1.1]{IKO20}, where they used the assumptions $T_K \neq \varphi$ and $h_K^+=1$ to prove it. However, 
			Theorem~\ref{main result1 for (p,p,2)} is independent of these assumptions.

			\section{Steps to prove Theorem~\ref{main result1 for (p,p,2)}}	
			\label{section for the proof of main result1 for (p,p,2)}
			For any non-trivial primitive solution $(a, b, c) \in \mcO_K^3$ to the equation~\eqref{p,p,2} of exponent $p$, consider the Frey  curve
			\begin{equation}
				\label{Frey curve for (p,p,2)}
				E=E_{a,b,c} : Y^2 = X^3 + 4cX^2 + 4a^pX,
			\end{equation}
			where $c_4=2^6(4c^2-3a^p)=2^6(a^p+4b^p),\ \Delta_E=2^{12}(a^2b)^p$ and $ j_E=2^6\frac{(a^p+4b^p)^3}{(a^2b)^p}$.
			
			\subsection{Modularity of the Frey curve}
			We now study the modularity of the Frey curve $E$ given by \eqref{Frey curve for (p,p,2)} over $K$.
			\begin{thm}
				\label{modularity of Frey curve (p,p,2)}
				Let $(a,b,c) \in W_K'$ be a non-trivial primitive solution to the equation~\eqref{p,p,2} of exponent $p$ and let $E := E_{a,b,c}$ be the associated Frey curve as in \eqref{Frey curve for (p,p,2)}. Then there exists a constant $A_K$ (depends on $K$) such that for primes $p >A_K$, $E/K$ is modular.
			\end{thm}
			\begin{proof}
				The proof of this theorem is similar to the proof of Theorem~\ref{modularity of Frey curve x^p+y^p=2^rz^p over K}, except that, 
				here, $j_E=2^6 \frac{(a^{p}+4b^p)^3}{(a^2b)^{p}}$,  $j(\lambda)=2^6\frac{(4\lambda -1)^3}{\lambda}$ for $\lambda= -\frac{b^p}{a^p}$. So, there exists 
				$\lambda_1, \lambda_2, ..., \lambda_m \in K$ such that $E$ is modular for all $\lambda \notin\{\lambda_1, \lambda_2, ..., \lambda_m\}$. 
				If $\lambda= \lambda_k$ for some $k \in \{1, 2, \ldots, m \}$, then $\left(\frac{b}{a} \right)^p=-\lambda_k$.
				The above equation determines $p$ uniquely and denotes it $p_k$.
				If not, let $p \neq q$ are primes such that $\left(\frac{b}{a} \right)^p=\left(\frac{b}{a} \right)^q$, which means $\left(\frac{b}{a}\right)$ is a root of unity. Since $K$ is totally real, we get $b=\pm a$. Again, since $(a,b,c)$ is primitive, $a=\pm1$ and $b=\pm1$, which is not possible since $(a,b,c) \in W_K'$. Hence the proof follows by taking $A_K=\max \{p_1,...,p_m\}$.
			\end{proof}    

			\subsection{Reduction type at odd primes}
			%
			
			The following lemma is an analog of I\c{s}ik, et al. (cf. 
			\cite[Lemma 3.2]{IKO20}) and Lemma~\ref{reduction away from S}, which specifies the types of reduction of the Frey curve $E$ given in \eqref{Frey curve for (p,p,2)} at primes $\mfq$ away from $S_K$.
			\begin{lem}
				\label{reduction away from S_K (p,p,2)}
				Let $(a,b,c) \in \mcO_K^3$ be a non-trivial primitive solution to the equation~\eqref{p,p,2} of exponent $p$ and let $E$ be the Frey curve in \eqref{Frey curve for (p,p,2)}. Then 
				at all primes $\mfq \in P$ away from $S_K$, $E$ is minimal, semi-stable and satisfies $p | v_\mfq(\Delta_E)$. Let $\mfn$ be the conductor of $E$ and $\mfn_p$ be as in \eqref{conductor of elliptic curve}. Then
				\begin{equation}
					\label{conductor of E (p,p,2)}
					\mfn=\prod_{\mfP \in S_K}\mfP^{r_\mfP} \prod_{\mfq|ab,\ \mfq \in P \setminus S_K }\mfq,\ \mfn_p=\prod_{\mfP \in S_K}\mfP^{r_\mfP'},
				\end{equation}
				where $0\leq r_\mfP' \leq r_\mfP \leq 2+6v_\mfP(2)$. 
			\end{lem}

			The following lemma describes the type of reduction of the Frey curve $E=E_{a,b,c}$ at $\mfP \in S_K$ and $(a,b,c)\in W_K'$. 
			More precisely:
			\begin{lem}
				\label{reduction on T_K and U_K on W_k' for (p,p,2)}
				\label{reduction on T_L and U_L on W_k' for (p,p,2)}
				Suppose $(a,b,c)\in W_K'$ is a non-trivial primitive solution to the equation~\eqref{p,p,2} of exponent $p$ and let $E := E_{a,b,c}$ be the associated Frey curve as in \eqref{Frey curve for (p,p,2)}. If $\mfP \in S_K$ and $p >  6v_\mfP(2)$, then $p | \#\bar{\rho}_{E,p}(I_\mfP)$. The same conclusion also holds for $\mfP \in S_L$.
		\end{lem} 
		\begin{proof}
			Suppose $\mfP \in S_K$ and $(a,b,c)\in W_K'$. By the definition of $W_K'$, either $\mfP | a$ or $\mfP |b$. Recall that $ j_E=2^6\frac{(a^p+4b^p)^3}{(a^2b)^p}$. If $\mfP |a$, then $\mfP \nmid b$ because $(a,b,c)$ is primitive. So $v_\mfP(j_E)= 6v_\mfP(2)+3v_\mfP(a^p+4b^p)-2pv_\mfP(a)$. Since $p >  6v_\mfP(2) $, $v_\mfP(j_E)=12v_\mfP(2)-2pv_\mfP(a)<0$ and $p \nmid v_\mfP(j_E)$. The same argument holds even if $\mfP |b$. The proof now follows from Lemma~\ref{criteria for potentially multiplicative reduction}.
		\end{proof}

		\subsection{Proof of Theorem~\ref{main result1 for (p,p,2)}}
        The proof of this theorem depends on the following auxiliary result..
		\begin{thm}
			\label{auxilary result to main result1 (p,p,2)}
			Let $K$ be a totally real field. Then there is a constant $B_K$ (depends on $K$) such that the following hold. Let $(a,b,c) \in W_K'$ be a non-trivial primitive solution to the equation~\eqref{p,p,2} of exponent $p > B_K$ and let $E := E_{a,b,c}$ be the associated Frey curve as in \eqref{Frey curve for (p,p,2)}. Then there exists an elliptic curve $E^\prime$ over K such that:
			\begin{enumerate}
				\item $E^\prime/K$ has good reduction away from $S_K$ and has a non-trivial $2$-torsion point;
				\item $\bar{\rho}_{E,p} \sim\bar{\rho}_{E^\prime,p}$ and  $v_\mfP(j_{E^\prime})<0$ for $\mfP \in S_K$.
			\end{enumerate}
		\end{thm} 
		 		\begin{proof}
           Arguing as in the proof of Theorem~\ref{auxilary result x^p+y^p=2^rz^p}, there exists a constant $B_K$ (depending on $K$) and an elliptic curve $E_f/K$ of conductor $\mfn_p$ such that 
           $\bar{\rho}_{E,p} \sim \bar{\rho}_{E_f,p}$ for all primes $p>B_{K}$.
           
           	\begin{enumerate}
           	\item Since the conductor of $E_f$ is $\mfn_p$ given in \eqref{conductor of E (p,p,2)}, $E_f$ has good reduction away from $S_K$. After enlarging $B_K$ by an efficient amount and by possibly replacing $E_f$ with an isogenous curve, say $E^\prime$, we will find that $E^\prime/ K$ has a non-trivial $2$-torsion point (cf.~\cite[page 1247]{M22} for more details). Also, $E^\prime$ has good reduction away from $S_K$.
           	
           	\item Since $E_f$ is isogenous to $E^\prime$ and $\bar{\rho}_{E,p} \sim \bar{\rho}_{E_f,p}$, we have $\bar{\rho}_{E,p} \sim\bar{\rho}_{E^\prime,p}$. So by Lemma~\ref{reduction on T_K and U_K on W_k' for (p,p,2)}, we have $p | \# \bar{\rho}_{E,p}(I_\mfP)=\# \bar{\rho}_{E^\prime,p}(I_\mfP)$ for any $\mfP \in S_K$. 
           	This gives $v_\mfP(j_{E^\prime})<0$, by Lemma~\ref{criteria for potentially multiplicative reduction}.
           \end{enumerate}
           
		\end{proof}  
	 
		We will now give a proof of  Theorem~\ref{main result1 for (p,p,2)} and it is similar to that of~\cite[Theorem 1.1]{IKO20}.

		\begin{proof}[Proof of Theorem~\ref{main result1 for (p,p,2)}]
			Let $B_K$ be as in Theorem~\ref{auxilary result to main result1 (p,p,2)} and let $(a,b,c)\in W_K'$ be a non-trivial primitive solution to the equation~\eqref{p,p,2} of exponent $p>B_K$.
         By Theorem~\ref{auxilary result to main result1 (p,p,2)}, there exists an elliptic curve $E^\prime/K$ having a non-trivial $2$-torsion point. Hence $E^\prime/K$ has a model of the form 
         \begin{equation}
         	E^\prime: y^2=x^3+cx^2+dx
         \end{equation}
         for some $c,d \in K$. Let $L$ be the splitting field of $x^3+cx^2+dx$ over $K$. Then either $E^\prime$ has full $2$-torsion over $K$, or $E^\prime$ has full $2$-torsion over $L$.
     
     Suppose $E^\prime$ has full $2$-torsion over $K$. Now, arguing as in the proof of Theorem~\ref{main result for x^p+y^p=2^rz^p} and by \eqref{assumption for main result1 for T_K for (p,p,2)}, we get $v_\mfP(j_{E^\prime}) \geq 0$  for some $\mfP \in S_K$, which is a contradiction to Theorem~\ref{auxilary result to main result1 (p,p,2)}.
     
     Suppose $E^\prime$ has full $2$-torsion over $L$. By \cite[Proposition 3.1]{K19}, $L=K(\sqrt{a})$ for some $a \in K(S_K,2)$. We now show that $E^\prime/L$ continues to have the same properties as in Theorem~\ref{auxilary result to main result1 (p,p,2)} over $L$. Namely;
     \begin{itemize}
      \item Since $E^\prime$ has good reduction away from $S_K$, $E^\prime/L$ has good reduction away from $S_L$ and has full $2$-torsion over $L$.  
      \item By Theorem~
      \ref{auxilary result to main result1 (p,p,2)}, $\bar{\rho}_{E,p} \sim\bar{\rho}_{E^\prime,p}$.
      By Lemma~\ref{reduction on T_K and U_K on W_k' for (p,p,2)}, we have $p | \# \bar{\rho}_{E,p}(I_\mfP)=\# \bar{\rho}_{E^\prime,p}(I_\mfP)$ for any $\mfP \in S_L$. So by Lemma~\ref{criteria for potentially multiplicative reduction}, $v_\mfP(j_{E^\prime})<0$ for $\mfP \in S_L$.
     \end{itemize}
      Arguing as in the proof of Theorem~\ref{main result for x^p+y^p=2^rz^p} and by \eqref{assumption for main result1 for T_L for (p,p,2)}, we get $v_\mfP(j_{E^\prime}) \geq 0$ for some $\mfP \in S_L$. This is a contradiction
      to Theorem~\ref{auxilary result to main result1 (p,p,2)} over $L$.
     	\end{proof}

        \section{Local criteria for the solutions of the Diophantine equation $x^p+y^p=2^rz^p$ over $K$}
        \label{section for local criteria}
        
        In this section, we give several purely local criteria for $K$ where the equation~\eqref{x^p +y^p = 2^rz^p} of exponent $p$ has no asymptotic solution in $\mcO_K^3$, in $W_K$. First, we give local criteria for $K$ of odd degree such that the equation~\eqref{x^p +y^p = 2^rz^p} of exponent $p$ with $r=2,3$ has no asymptotic solution in $\mcO_K^3$. More specifically:
		
		\begin{prop}
			\label{local criteria1 for no solution in K}
			Let $n= [K: \Q]$ and $l>5$ be a prime such that $(n, l-1) = 1$. Suppose that $l$ totally ramifies and $2$ is inert in $K$. Then the equation~\eqref{x^p +y^p = 2^rz^p} of exponent $p$ with $r=2,3$ has no asymptotic solution in $\mcO_K^3$.
			
		\end{prop}
		For example, one can take $K=\Q(\alpha)$, where $\alpha$ satisfies the minimal polynomial $x^3 - x^2 + 1$. In $K$,  $2$ is inert, $23$ is totally ramified in $K$ and $(3, 23-1)=1$.
		

		\begin{proof}[Proof of Proposition~\ref{local criteria1 for no solution in K}]
			Let $(\lambda, \mu)$ be a solution of the $S_K$-unit equation~\eqref{S_K-unit solution}. Let $\mfP\in U_K$ be the unique prime ideal over $2$ and let $m_{\lambda, \mu}:=\max \left\{|v_\mfP(\lambda)|,|v_\mfP(\mu)| \right\}$. Now, we show that $m_{\lambda, \mu}<2 v_\mfP (2)=2$.
			If not, let $m_{\lambda, \mu} \geq 2 v_\mfP (2)=2$. By~\cite[Lemma 3.1]{FKS21}, there exists a solution $(\lambda', \mu')$ of \eqref{S_K-unit solution} satisfying $\lambda' \in \mcO_K, \ \mu' \in \mcO_K^\ast$ with $m_{\lambda, \mu}=m_{\lambda', \mu'}$, i.e., 
			$v_\mfP (\mu')=0$ and $v_\mfP (\lambda')=m_{\lambda', \mu'}$. Therefore, $\lambda' \equiv 0 \pmod {\mfP^2}$, which gives $\lambda' \equiv 0 \pmod 4$ and hence $\mu' \equiv 1 \pmod 4$. Thus $\text{Norm}_{K/ \Q}(\mu') \equiv 1 \pmod 4$. Since $\mu'$ is a unit, $\mathrm{Norm}_{K/ \Q}(\mu')=1$.
			Now, arguing as in the proof of~\cite[Lemma 4.1]{FKS21}, we get
			$\mathrm{Norm}_{K/ \Q}(\mu')=-1$, which is a contradiction.
			So, $m_{\lambda, \mu} < 2 v_\mfP (2)=2$.
			Hence, $m_{\lambda, \mu}=0$ or $1$. If $m_{\lambda, \mu}=0$, then $(\lambda, \mu)$ is a solution of the unit equation $\lambda+ \mu  =1$, which contradicts~\cite[Theorem 4]{FKS21}. Thus, $m_{\lambda, \mu}=1=v_\mfP (2)$. 
			Now, the proof of Proposition~\ref{local criteria1 for no solution in K} follows from Corollary~\ref{cor to main result1 for x^p+y^p=2^rz^p}.
		\end{proof}

		\begin{prop}
			\label{local criteria2 for no solution in K}
			Let $n=[K: \Q]$ be odd and $3 \nmid n$. Suppose $2$ is inert and $3$ totally splits in $K$. Then the conclusion of Proposition~\ref{local criteria1 for no solution in K} holds.
            
		\end{prop}
		
		\begin{proof}
           Let $\mfP\in U_K$ be the unique prime ideal over $2$. 
           Suppose, there exists a solution $(\lambda, \mu)$ to $S_K$-unit equation~\eqref{S_K-unit solution} such that $m_{\lambda, \mu}\geq 2 v_\mfP (2)$. Then arguing as in the proof of Proposition~\ref{local criteria1 for no solution in K}, we get a solution $(\lambda', \mu')$ to $S_K$-unit equation~\eqref{S_K-unit solution} with $\lambda' \in \mcO_K, \ \mu' \in \mcO_K^\ast$ such that $\lambda' \equiv 0 \pmod {4}, \mu' \equiv 1 \pmod 4$. This imply $\mathrm{Norm}_{K/ \Q}(\mu')=1$. 
           By~\cite[Lemma 5.1]{FKS21}, we have $\mu' \equiv -1 \pmod3$. So $\mathrm{Norm}_{K/ \Q}(\mu') \equiv (-1)^n=-1 \pmod3$ and hence $\mathrm{Norm}_{K/ \Q}(\mu')=-1$, which is a contradiction. Thus, every solution $(\lambda, \mu)$ to the $S_K$-unit equation~\eqref{S_K-unit solution} satisfies $m_{\lambda, \mu}<2 v_\mfP (2)=2$. Hence, $m_{\lambda, \mu}=0$ or $1$. If $m_{\lambda, \mu}=0$, then $(\lambda, \mu)$ is a solution of the unit equation $\lambda+ \mu  =1$, which contradicts
           ~\cite[Theorem 1]{T20}. Note that $3 \nmid n$ and $3$ totally splits in $K$. 
           So, $m_{\lambda, \mu}=1=v_\mfP(2)$. Now, the proof of Proposition follows from Corollary~\ref{cor to main result1 for x^p+y^p=2^rz^p}.
		\end{proof}
		
		We give local criteria for $K$ with $[K:\Q] \equiv 1 \pmod 2$
		such that the equation~\eqref{x^p +y^p = 2^rz^p} of exponent $p$ has no asymptotic solution in $W_K$. These criteria are an analogue of~\cite[Theorems 1,3]{FKS21} and their proofs are an application of Theorem~\ref{main result for x^p+y^p=2^rz^p}.
        \begin{cor}
			\label{local criteria1 for no solution in W_K}
			Suppose one of the hypothesis holds:
			\begin{enumerate}
				\item Suppose $n=[K:\Q]$ and $l>5$ is a prime number such that $(n, l-1) = 1$. Assume $2$ is either inert or totally ramified in $K$ and $l$ totally ramifies in $K$.
				\item Suppose $[K:\Q]$ is odd, $2$ is either inert or totally ramified in $K$ and $3$ totally splits in 
				$K$.
			\end{enumerate}
			Then the equation~\eqref{x^p +y^p = 2^rz^p} of exponent $p$ has no asymptotic solution in $W_K$.
		\end{cor}
		\begin{proof}
		Let $\mfP\in S_K$ be the unique prime over $2$.
		    For part(1) (resp., part(2)), by arguing as in the proof of Proposition~\ref{local criteria1 for no solution in K} (resp., Proposition~\ref{local criteria2 for no solution in K}), every solution $(\lambda, \mu)$ to the $S_K$-unit solution \eqref{S_K-unit solution} satisfies $m_{\lambda, \mu}<2 v_\mfP(2)$. 
        Now, the proof of the corollary follows from Theorem~\ref{main result for x^p+y^p=2^rz^p}.
\end{proof}
Note that, in part (2), we allow $2$ to be inert or totally ramified in $K$, whereas in~\cite[Theorem 3]{FKS21}, $2$ is totally ramified. 
Finally, we give local criteria for real quadratic fields $K$ such that the equation~\eqref{x^p +y^p = 2^rz^p} of exponent $p$ has no asymptotic solution in $W_K$ (resp., in $\mcO_K^3$) and these are analogue of~\cite[Theorems 1,2]{FS15}.

\begin{cor}
	\label{local criteria2 quadratic for no solution in W_K}
	Let $K=\Q(\sqrt{d})$. Suppose $d \geq2 $ is a square-free integer satisfy one of the following conditions:
	\begin{enumerate}
		\item $d \equiv 3 \pmod 8$;
		\item $d \equiv 5 \pmod 8$;
		\item $d \equiv 6 \text{ or } 10 \pmod {16}$;
		\item $d \equiv 2 \pmod {16}$ and $d$ has some prime divisor $q \equiv 5 \text{ or } 7 \pmod 8$;
		\item $d \equiv 14 \pmod {16}$ and $d$ has some prime divisor $q \equiv 3 \text{ or } 5 \pmod 8.$
	\end{enumerate} 
	Then the equation~\eqref{x^p +y^p = 2^rz^p} of exponent $p$ has no asymptotic solution in $W_K$.
\end{cor}
\begin{proof}
	Let $\mfP \in S_K$. By~\cite[Table 1 in \S6]{FS15}, the $S_K$-unit equation~\eqref{S_K-unit solution} has only irrelevant solutions $(2,-1), (-1,2)$ and $ (\frac{1}{2}, \frac{1}{2})$. In these cases, $\max \left\{|v_\mfP(\lambda)|,|v_\mfP(\mu)| \right\}= v_\mfP(2)$. Since 
	$\mfP$ satisfies~\eqref{assumption for main result x^p+y^p=2^rz^p},
	the proof of the corollary follows from Theorem~\ref{main result for x^p+y^p=2^rz^p}.
\end{proof}

The proof of Corollary~\ref{local criteria2 quadratic for no solution in W_K} uses the hypothesis~\eqref{assumption for main result x^p+y^p=2^rz^p} for $\mfP \in S_K$, whereas~\cite[Theorem 1]{FS15} uses the hypothesis for $\mfP \in T_K$. This freedom allows us to include the fields $\Q(\sqrt{d})$ for $d\equiv 5 \pmod 8$ in Corollary~\ref{local criteria2 quadratic for no solution in W_K} when compared with the list of fields in~\cite[Theorem 1]{FS15}. In this case, $S_K \neq \varphi$, but $T_K=\varphi$.
	
		\begin{cor}
		\label{local criteria2 quadratic for no solution in K}
		Let $d\geq 2$ be a square-free integer satisfying one of the congruences in Corollary~\ref{local criteria2 quadratic for no solution in W_K}. Suppose that Conjecture~\ref{ES conj} holds over $K=\Q(\sqrt{d})$. Then the equation~\eqref{x^p +y^p = 2^rz^p} of exponent $p$ with $r=2,3$ has no asymptotic solution in $\mcO_K^3$.
	\end{cor}
	\begin{proof}
        Since $[\Q(\sqrt{d}) : \Q] =2$, $v_\mfP(2) = 1$ or $2$, which implies $S_K =U_K$. Let $\mfP \in S_K$. By \cite[Table 1 in \S6]{FS15}, the $S_K$-unit equation~\eqref{S_K-unit solution} has only irrelevant solutions $(2,-1), (-1,2)$ and $ (\frac{1}{2}, \frac{1}{2})$. In these cases, $\max \left\{|v_\mfP(\lambda)|,|v_\mfP(\mu)| \right\}=v_\mfP(2)$. Now, the proof
		of the corollary follows from Corollary~\ref{cor to main result1 for x^p+y^p=2^rz^p}.
	\end{proof}


\begin{thebibliography}{abc9999}
\bibitem[Tri]{T20}
										Triantafillou, Nicholas.
										The unit equation has no solutions in number fields of degree prime to $3$ where $3$ splits completely. To appear in Proceedings of the AMS, Series B.
\end{thebibliography}
								\end{document}